\documentclass[preprint,12pt]{elsarticle}

\usepackage{amssymb}
\usepackage{url}
\usepackage{amsthm}
\usepackage{amsmath}

\newtheorem{dfn}{Definition}
\newtheorem{lem}{Lemma}

\newtheorem{thm}{Theorem}
\newtheorem{prp}{Property}
\newtheorem{pro}{Proposition}
\newtheorem{cor}{Corollary}

\journal{}

\begin{document}

\begin{frontmatter}

\title{A new lower bound for doubly metric dimension and related extremal differences}

\author{Jozef Kratica \fnref{mi}}
\ead{jkratica@mi.sanu.ac.rs}
\author{Vera Kova\v{c}evi\'c-Vuj\v{c}i\'c \fnref{fon}}
\ead{vera.vujcic@alumni.fon.bg.ac.rs}
\author{Mirjana \v{C}angalovi\'c \fnref{fon}}
\ead{mirjana.cangalovic@alumni.fon.bg.ac.rs}

 \address[mi]{Mathematical Institute, Serbian Academy of Sciences and Arts, Kneza Mihaila 36/III, 11 000 Belgrade, Serbia}

 \address[fon]{Faculty of Organizational Sciences, University of  Belgrade, Jove Ili\'ca 154, 11000 Belgrade, Serbia}   

\begin{abstract}
In this paper a new graph invariant based on the minimal hitting set problem is introduced.
It is shown that it represents a tight lower bound for the doubly metric dimension of 
a graph. Exact values of new invariant for paths, stars, complete graphs and
complete bipartite graph are obtained. The paper analyzes some tight bounds 
for the new invariant in general case. Also several extremal differences 
between some related invariants are determined.
\end{abstract}


\end{frontmatter}

%% \linenumbers

\section{Introduction}

The metric dimension problem was introduced independently by Slater in 1975 \cite{metd1} and 
Harary and Melter in 1976 \cite{metd2}. Given a simple connected undirected graph $G = (V,E)$,
$d(u,v)$ denotes the distance between vertices $u$ and $v$, i.e. the length of a shortest $u-v$ path.
Vertex $w \in V(G)$ of graph $G$ is said to resolve two vertices $u,v \in V(G)$ if $d(u,w) \ne d(v,w)$.

Set $S \subseteq V(G)$ is a resolving set of $G$ if any pair of distinct vertices of $G$ are resolved by some vertex from $S$.
A metric basis of $G$ is a resolving set of minimal cardinality. The metric dimension of $G$, denoted 
by $\beta(G)$, is the cardinality of its metric basis. The problem of computing the metric dimension of
an arbitrary graph is NP-hard \cite{metdnp}. 

In 2007, Caceres et al. \cite{cac07} defined the notion of a doubly resolving set as follows.
Vertices $x,y \in V(G)$ are said to doubly resolve vertices $u,v \in V(G)$ if $d(u,x)-d(u,y) \ne d(v,x)-d(v,y)$. Set $D \subseteq V(G)$
is a doubly resolving set of $G$ if every two distinct vertices of $G$ are doubly resolved by some two vertices of $D$. The doubly metric dimension
of $G$, denoted by $\psi(G)$, is the minimal cardinality of doubly resolving sets of $G$. The problem of finding the doubly metric dimension of an arbitrary
graph $G$ is also NP-hard \cite{kra09}. 

The concepts of edge and mixed metric dimensions were introduced by Kelenc et. al \cite{edim1,mdim1}. 
For the sake of simplicity, edge $e \in E(G)$ with endpoints $u$ and $v$ will be denoted by $e=uv$.
The distance between edge $e=uv \in E(G)$ and vertex $w \in V(G)$, denoted by $d(e,w)$, is defined as $d(e,w) = min \{ d(u,w), \, d(v,w)\}$. 
Vertex $w$ resolves two edges $e_1$ and $e_2$ if $d(e_1,w) \ne d(e_2,w)$. Set  $N \subseteq V(G)$ is an edge resolving set if 
for any pair of edges from $E(G)$ there is some vertex in $N$ that resolves them. The edge metric dimension of $G$, denoted by
$\beta_E(G)$, is the minimum cardinality of edge resolving sets of $G$. 

Finally, the concept of mixed metric dimension unifies metric and edge metric dimensions in the following way. Vertex $w$
resolves two items $a,b \in V(G) \bigcup E(G)$ if $d(a,w) \ne d(b,w)$. Mixed resolving set $M \subseteq V(G)$ is defined as a set 
such that for any pair of items from $V(G) \bigcup E(G)$ there is some vertex in $M$ that resolves them. Following earlier definitions, 
the mixed metric dimension of $G$, denoted by $\beta_M(G)$ is defined as the minimal cardinality of mixed resolving sets of $G$.
Finding $\beta_E(G)$ and $\beta_M(G)$ are both NP-hard in general case \cite{edim1,mdim1}.

In the sequel the following definitions will be used. 
Let $\Delta(G)$ denote the maximal degree of vertices in graph $G$. 
Additionally, $N(v) = \{w \in V(G) | vw \in E(G)\}$ and $N[v] = \{v\} \bigcup N(v)$
are open and closed neighbourhood of $v$, respectively.

\begin{dfn} \label{wset} \mbox{\rm(\cite{bal09})} For an arbitrary edge $e = uv \in E(G)$ 
$$W_{uv}=\{w \in V(G)| d(u,w)<d(v,w)\}$$
and 
$$W_{vu}=\{w \in V(G)| d(v,w)<d(u,w)\}$$
\end{dfn}

Complements of sets $W_{uv}$ and $W_{vu}$, denoted by $\overline{W_{uv}}$ and $\overline{W_{vu}}$,
are the following sets:
$\overline{W_{uv}}=\{w \in V(G)| d(u,w) \ge d(v,w)\}$
and 
$\overline{W_{vu}}=\{w \in V(G)| d(v,w) \ge d(u,w)\}$.  \\

The set of vertices on equal distances from $u$ and $v$ is denoted in the literature by $_u{W_v}$ (\cite{bal09}).
Note that $_u{W_v} = \overline{W_{uv}} \setminus W_{vu} = \overline{W_{vu}} \setminus W_{uv}$.
Also, $\overline{W_{uv}} \bigcup \overline{W_{vu}} = V(G)$.

The following definition introduces the well-known concept of hitting sets \cite{hit}.

\begin{dfn} \label{hs1} For a given set $S$ and family $\mathop{F} =\{S_1,...,S_k\}$, $S_i \subseteq S$, $\bigcup \limits_{i = 1}^k S_i=S$,   
a hitting set $H \subseteq S$ of family $F$ is a set which has a non-empty intersection with each $S_i$,
i.e. $(\forall i \in \{1,...,k\}) H \bigcap S_i \ne \emptyset$. 
\end{dfn}

The minimal hitting set problem is to find a hitting set of the minimal cardinality. This problem is equivalent to the set covering problem and is known to be NP-hard \cite{np}.

\begin{dfn} \mbox{\rm(\cite{mil21})} \label{hsl} Value $mhs_<(G)$ is defined to be the minimal cardinality of hitting sets of family $\{W_{uv}, W_{vu} | uv \in E(G) \}$.\end{dfn}

It is easy to see that the following property holds:

\begin{prp} \label{mhsg2} $mhs_<(G) \ge 2$.\end{prp}
\begin{proof} Let $uv$ be an arbitrary edge of $G$. Since $W_{uv} \cap W_{vu} = \emptyset$, a hitting set of $\{W_{uv}, W_{vu} \}$
has at least two elements, one from $W_{uv}$ and one from $W_{vu}$. Therefore, a hitting set of family 
$\{W_{uv}, W_{vu} | uv \in E(G) \}$ also has at least two elements. By Definition \ref{hsl},
it follows that $mhs_<(G) \ge 2$.
\end{proof}

The following proposition was proved in \cite{mil21}:

\begin{pro} \label{mmetd} \mbox{\rm(\cite{mil21})} Let $G$ be a connected graph, $uv\in E(G)$ an arbitrary edge and $M$ a mixed resolving set of $G$. Then
 \begin{itemize}
  \item [(i)] $W_{uv}\cap M\neq \emptyset;$
   \item [(ii)] $W_{vu}\cap M\neq \emptyset$.
  \end{itemize}
\end{pro} 

\begin{cor} \label{mmetd2} \mbox{\rm(\cite{mil21})} For every connected graph $G$ it holds $\beta_M(G) \ge mhs_<(G)$.
\end{cor} 

In the sequel we will use the following results about $\psi(G)$, $\beta_E(G)$ and $\beta_M(G)$.

\begin{pro} \mbox{\rm(\cite{cac07})} \label{dou1a} Let $G$ be any graph of order $n$, then $2 \le \psi(G) \le n-1$.\end{pro} 

\begin{pro} \mbox{\rm(\cite{cac07})} \label{dou2a} Let $K_n$ be a complete graph of order $n$, then $\psi(K_n) = max\{2,n-1\}$.\end{pro} 

\begin{pro} \mbox{\rm(\cite{cac07})} \label{tree} Let $T$ be any tree with $l(T)$ leaves, then $\psi(T) = l(T)$.\end{pro} 

\begin{cor} \label{psit} For any path $P_n$ it holds $\psi(P_n) = 2$.
\end{cor} 

\begin{pro} \mbox{\rm(\cite{cac07})} \label{dcyc} For any cycle $C_n$, $n \ge 3$, it holds $$\psi(C_n) = \begin{cases}
2, & n \, is \, odd \\
3, & n \, is \, even\
\end{cases}$$
 \end{pro}

\begin{pro} \mbox{\rm(\cite{edim1})} \label{edim1} For $n \ge 2$ it holds $\beta_E(P_n) = \beta(P_n) = 1$,
$\beta_E(C_n) = \beta(C_n) = 2$, $\beta_E(K_n) = \beta(K_n) = n-1$. Moreover, $\beta_E(G)=1$ if and only if $G$ is a path.
\end{pro}

\begin{pro} \mbox{\rm(\cite{edim1})} \label{edim2} Let $G$ be a connected graph. Then $\beta_E(G) \geq \lceil log_2 \Delta(G) \rceil$
\end{pro}

\begin{pro} \label{mpath} \mbox{\rm(\cite{mdim1})} For a connected graph $G$, it holds $2 \leq \beta_M(G) \leq n$. Moreover, $\beta_M(G)=2$ if and only if $G$ is a path.
 \end{pro} 

\begin{pro} \label{mbip} \mbox{\rm(\cite{mdim1})} For any complete bipartite graph $K_{r,t}$, $r,t \ge 2$, it holds $$\beta_M(K_{r,t}) = \begin{cases}
r+t-1, & r=2 \vee t=2 \\
r+t-2, & otherwise \\
\end{cases}$$
 \end{pro}

\begin{dfn} \label{mngd} Let $v$ be a vertex of graph $G$. A vertex $u \in N(v)$ is called a maximal neighbour of $v$
if $N[v] \subseteq N[u]$. $G$ is a maximal neighbour graph if each vertex $v \in V(G)$ has a maximal neighbour.
\end{dfn}

\begin{thm} \label{mngt} \mbox{\rm(\cite{mdim1})} Let $G$ be any graph of order $n$. Then $\beta_M(G) = n$ if and only if 
$G$ is a maximal neighbour graph. \end{thm} 

The following definition introduces a new graph invariant which will be analyzed in this paper.

\begin{dfn} \label{hsle} Value $mhs_{\le}(G)$ is defined to be the minimal cardinality of hitting sets of family $\{\overline{W_{uv}}, \overline{W_{vu}} | uv \in E(G) \}$.\end{dfn}

The paper also deals with extremal differences of graph invariants which are defined as follows.

\begin{dfn} \label{dif1} Let $\xi_1(G)$ and $\xi_2(G)$ be two graph invariants. Then extremal difference $(\xi_1-\xi_2)(n)$ is defined as the maximal value of difference 
$\xi_1(G)-\xi_2(G)$ for all connected graphs $G$ of order $n$. 
\end{dfn}

It should be noted that the minimal value of difference $\xi_1(G)-\xi_2(G)$ is equal to $-(\xi_2-\xi_1)(n)$.
Also, only non-trivial extremal differences are in cases when $n \ge 3$, since path $P_2$ is only connected graph of order 2, 
so it holds $(\xi_1-\xi_2)(2) = \xi_1(P_2)-\xi_2(P_2)$.

The paper is organized as follows. In Section 2 it is shown that $mhs_{\le}(G)$ is a new lower bound for $\psi(G)$.
Exact values of the new invariant are obtained for some special classes of graphs. Also, some tight bounds for $mhs_{\le}(G)$
in general case are derived. Section 3 provides the exact values of extremal differences between invariants 
$\psi(G)$ and $mhs_{\le}(G)$, $mhs_<(G)$ and $mhs_{\le}(G)$, $\beta_M(G)$ and $mhs_<(G)$, as well as 
some upper and lower bounds for extremal difference between $\psi(G)$ and $\beta_E(G)$. 
Section 4 gives concluding remarks about obtained results and gives some
directions for future work. 

\section{A new lower bound for $\psi(G)$}

The following theorem and its corollary give a new lower bound for $\psi(G)$.

\begin{thm} \label{doub} Let $G$ be a connected graph, $uv\in E(G)$ an arbitrary edge and $S$ a doubly resolving set of $G$. Then
 \begin{itemize}
  \item [(i)] $\overline{W_{uv}}\cap S\neq \emptyset;$
   \item [(ii)] $\overline{W_{vu}}\cap S\neq \emptyset$.
  \end{itemize}
\end{thm} 
\begin{proof} Let us suppose that the statement of Theorem \ref{doub} does not hold, i.e. there exists an edge $uv \in E(G)$
such that\\
$\overline{W_{uv}}\cap S = \emptyset$ or $\overline{W_{vu}}\cap S = \emptyset$. \\
In the first case it follows that $d(u,w) < d(v,w)$ for each $w \in S$. Since $uv \in E(G)$ it follows that $d(v,w) = d(u,w)+1$,
i.e. $d(v,w)-d(u,w) = 1$ for each $w \in S$. This implies that vertices $u$ and $v$ are not doubly resolved by any pair of vertices from $S$, 
which is a contradiction.\\
In the case when $\overline{W_{vu}}\cap S = \emptyset$ a similar contradiction can be derived.
\end{proof}

Theorem \ref{doub} implies that any doubly resolving set $S$ of $G$ is a hitting set of family $\{\overline{W_{uv}}, \overline{W_{vu}} | uv \in E(G) \}$.
Since $mhs_{\le}(G)$ is the minimal cardinality of hitting sets of this family, the following corollary of Theorem \ref{doub} holds.

\begin{cor} \label{doub2} For every connected graph $G$ it holds $\psi(G) \ge mhs_{\le}(G)$
\end{cor} 

In the following Proposition \ref{hitst} exact values of $mhs_{\le}(G)$ and $mhs_<(G)$ for some special types of graphs are derived.

The obtained values are used in Lemma \ref{hitl1} to prove that certain lower and upper bounds of $mhs_{\le}(G)$ and $mhs_<(G)$
for arbitrary connected graph $G$ are tight.

\begin{pro} \label{hitst} 
\begin{itemize}
\item [(i)] For path $P_n$, $n \ge 2$, it holds that \\ $mhs_{\le}(P_n) = mhs_<(P_n) = 2$; 
\item [(ii)] For star $S_n$, $n \ge 3$, it holds that $mhs_{\le}(S_n) = mhs_<(S_n) = n-1$;
\item [(iii)] For complete graph $K_n$, $n \ge 3$, it holds that $mhs_{\le}(K_n) = 2, mhs_<(K_n) = n$.
\item [(iv)] For complete bipartite graph $K_{2,m}$, $m \ge 2$, it holds that $mhs_{\le}(K_{2,m}) =  mhs_<(K_{2,m}) = 2$.
\end{itemize}
\end{pro}
\begin{proof}(i) Let $H$ be a hitting set of family $\{ W_{uv}, W_{vu} | uv \in E(P_n)\}$. Since 
$V(P_n) = \{ v_i  |  1 \le i \le n \}$ and  $E(P_n) =  \{ v_iv_{i+1}  |  1 \le i \le n-1 \}$, it is easy to see that
$W_{v_iv_{i+1}} = \{ v_j  |  1 \le j \le i \} $ and $W_{v_{i+1}v_i} = \{ v_j  |  i+1 \le j \le n \} $ implying 
$\overline{W_{v_iv_{i+1}}} = W_{v_{i+1}v_i}$ and $\overline{W_{v_{i+1}v_i}} = W_{v_iv_{i+1}}$. 
Therefore $H$ is also a hitting set of family $\{ \overline{W_{uv}}, \overline{W_{vu}} | uv \in E(P_n)\}$.
Since $W_{v_1v_2} = \{ v_1 \}$ and $W_{v_1v_2} \cap S \neq \emptyset$ 
it holds that $v_1 \in H$. In the same way, since $W_{v_nv_{n-1}} = \{ v_n \}$ and $W_{v_nv_{n-1}} \cap H \neq \emptyset$ 
it holds that $v_n \in H$. Therefore, $|H| \ge 2$. Since $v_1 \in W_{v_iv_{i+1}}$ and $v_n \in W_{v_{i+1}v_i}$,
$i=1,...,n-1$ the hitting set of the minimal cardinality is exactly $H = \{v_1, v_n \}$ and
$mhs_<(P_n) = mhs_{\le}(P_n) = 2$. \\
 (ii) Star $S_n$ is defined by $V(S_n) = \{ v_i  |  1 \le i \le n \}$,  $E(S_n) =  \{ v_1v_i  |  2 \le i \le n \}$. It is easy to see that
$W_{v_iv_1} = \{v_i\}$ and $W_{v_1v_i} = V(S_n) \setminus \{v_i\}$ for each $i=2,...,n$. Therefore, a hitting set $H$ of family 
$\{W_{v_iv_1}, W_{v_1v_i} | 2 \le i \le n \}$ satisfies $\{v_2, ..., v_n\} \subseteq H$. Since $W_{v_1v_i} \bigcap \{v_2, ..., v_n\} \ne \emptyset$ and
$W_{v_iv_1} \bigcap \{v_2, ..., v_n\} \ne \emptyset$ for $i=2,...,n$ it follows that $\{v_2, ..., v_n\}$ is a hitting set of the minimal cardinality, i.e.
$mhs_<(S_n) = n-1$. Furthermore, since $\overline{W_{v_iv_1}} = W_{v_1v_i}$ and $\overline{W_{v_1v_i}} = W_{v_iv_1}$ for $i=2,...,n$
it follows that $\{v_2, ..., v_n\}$ is also a hitting set of family $\{ \overline{W_{v_iv_1}}, \overline{W_{v_1v_i}} | 2 \le i \le n \}$ with the minimal cardinality and hence 
$mhs_{\le}(S_n) = n-1$. \\
(iii) It is easy to check that $W_{uv} = \{u\}$, $W_{vu} = \{v\}$ for each $uv \in E(K_n)$. Therefore $V(K_n)$ is a hitting set of family 
$\{W_{uv}, W_{vu} | uv \in E(K_n)\}$ with the minimal cardinality and $mhs_<(K_n) = n$. \\
Also,  $\overline{W_{uv}} = V(K_n) \setminus \{u\}$ and $\overline{W_{vu}} = V(K_n) \setminus \{v\}$, for each $uv \in E(K_n)$.
As $u \notin \overline{W_{uv}}$ and $v \notin \overline{W_{vu}}$, a set consisting of one vertex can not be a hitting set of family
$\{\overline{W_{uv}}, \overline{W_{vu}} | uv \in E(K_n)\}$, i.e.  $mhs_{\le}(K_n) \ge 2$. \\
Let $a,b \in V(K_n)$, $a \ne b$, be two arbitrrary distinct vertices. Then for each $uv \in E(K_n)$ there are three possible cases:\\
Case 1: If $u \ne a$ and $v \ne a$ then $a \in \overline{W_{uv}}$ and $a \in \overline{W_{vu}}$; \\
Case 2: If $u = a$ then $a \notin \overline{W_{av}}$ and $a \in \overline{W_{va}}$; \\
Case 3: If $v = a$ then $a \in \overline{W_{ua}}$ and $a \notin \overline{W_{au}}$. \\
However, $b \in \overline{W_{av}}$ and $b \in \overline{W_{au}}$. Therefore, $\{a,b\}$
is a hitting set of family $\{\overline{W_{uv}}, \overline{W_{vu}} | uv \in E(K_n)\}$ with 
the minimal cardinality and hence $mhs_{\le}(K_n) = 2$. \\
(iv) Complete bipartite graph $K_{2,m}$ is defined by $V(K_{2,m}) = V_1 \bigcup V_2$,  $V_1 = \{u_1,u_2\}$, $V_2 = \{v_i |1 \le i \le m\}$ and
$E(K_{2,m}) = \{u_iv_j|1 \le i \le 2, 1 \le j \le m\}$. From Property \ref{mhsg2} it follows that $mhs_<(K_{2,m}) \ge 2$. 
Next we prove that set $\{u_1,u_2\}$ is a hitting set of family $\{ W_{u_iv_j}, W_{v_ju_i} | 1 \le i \le 2, 1 \le j \le m \}$. \\
As $d(a,b)=1$ if $a \in V_1$ and $b \in V_2$ and $d(a,b)=2$ for $a,b \in V_1$ or $a,b \in V_2$
it follows that for each edge $u_iv_j \in E(K_{2,m})$ we have $W_{u_iv_j} = \{u_i\} \bigcup (V_2 \setminus \{v_j\})$,
$W_{v_ju_i} = \{v_j\} \bigcup (V_1 \setminus \{u_i\})$. Then $u_1 \in W_{u_1v_j}$,  $u_1 \notin W_{v_ju_1}$,
 $u_2 \in W_{u_2v_j}$,  $u_2 \notin W_{v_ju_2}$ for each $j=1,...,m$. However, $u_1 \in W_{v_ju_2}$ and 
 $u_2 \in W_{v_ju_1}$. Therefore, $\{u_1,u_2\}$ is a hitting set of the minimal cardinality and $mhs_<(K_{2,m}) = 2$. \\
As $\overline{W_{u_iv_j}} = W_{v_ju_i}$ and $\overline{W_{v_ju_i}} = W_{u_iv_j}$ for $i=1,2$ and $j=1,...,m$
then $\{u_1,u_2\}$ is a hitting set of family $\{\overline{W_{u_iv_j}}, \overline{W_{v_ju_i}} | 1 \le i \le 2, 1 \le j \le m \}$ with
the minimal cardinality and hence $mhs_{\le}(K_{2,m}) = 2$.
\end{proof}

According to Corollary \ref{psit}, Corollary \ref{doub2} and Proposition \ref{hitst}, new invariant $mhs_{\le}(G)$ represents a tight
lower bound for doubly metric dimension $\psi(G)$ as $\psi(P_n) = mhs_{\le}(P_n) = 2$. 

Next lemma gives tight lower and upper bounds for $mhs_{\le}(G)$ and $mhs_<(G)$
for arbitrary connected graph $G$.

\begin{lem} \label{hitl1} For $n \ge 3$ and graph $G$ of order $n$ it holds:
\begin{itemize}
\item [(i)] $mhs_{\le}(G) \le  mhs_<(G)$ 
\item [(ii)] $2 \le mhs_{\le}(G) \le  n-1$
\item [(iii)] $mhs_<(G) \le  n$
\end{itemize}
\end{lem}
\begin{proof}
(i) As $W_{vu} \subseteq \overline{W_{uv}}$ and $W_{uv} \subseteq \overline{W_{vu}}$ for each $uv \in E(G)$ it follows that $mhs_{\le}(G) \le  mhs_<(G)$  
This inequality is tight since by Propostion \ref{hitst} it holds $mhs_{\le}(P_n) = mhs_<(P_n) = 2$; \\
(ii) Let us prove that $mhs_{\le}(G) > 1$. If we suppose that $mhs_{\le}(G) = 1$ then there exist vertex $x \in V(G)$ such that $\{ x \}$ is a hitting set of family 
$\{\overline{W_{uv}}, \overline{W_{vu}} | uv \in E(G)\}$, i.e. $x \in \overline{W_{uv}}$ and $x \in \overline{W_{vu}}$ for each $uv \in E(G)$. As $G$ 
is connected there exists vertex $a \in V(G)$ such that $ax \in E(G)$. Since $\overline{W_{ax}} = \{w \in V(G) | d(a,w) \ge d(x,w)\}$ and
$\overline{W_{xa}} = \{w \in V(G) | d(x,w) \ge d(a,w)\}$ it follows that $x \in \overline{W_{ax}}$ but $x \notin \overline{W_{xa}}$,
which is a contradiction. Therefore, $mhs_{\le}(G) \ge 2$. This lower bound is tight since by Propostion \ref{hitst} it can be achieved for path $P_n$
or complete graph $K_n$.\\
From Corollary \ref{doub2} it is known that $mhs_{\le}(G) \le \psi(G)$, while by Propostion \ref{dou1a} from \cite{cac07} it is known that $2 \le \psi(G) \le n-1$,
implying that $mhs_{\le}(G) \le n-1$. By Propostion \ref{hitst} this upper bound is reached for star $S_n$. \\
(iii) As $W_{uv}$ and $W_{vu}$ are subsets of $V(G)$ for each $uv \in E(G)$, and $G$ is of order $n$, it follows $mhs_<(G) \le n$. 
By Propostion \ref{hitst} this upper bound is reached for complete graph $K_n$. 
\end{proof}

Is should be noted that lower bound $mhs_<(G) \ge 2$ from Property \ref{mhsg2} is tight for e.g. path $P_n$.

\section{Extremal differences}

The first task in this section is to find extremal diferences between doubly metric dimension
$\psi(G)$ and new invariant $mhs_{\le}(G)$. The complete answer is given in Theorem \ref{psimhs1}.

\begin{thm} \label{psimhs1} For $n \ge 3$ it holds 
\begin{itemize}
\item [(i)]  $(mhs_{\le} - \psi)(n) =0$
\item [(ii)] $(\psi - mhs_{\le})(n) = n-3$
\end{itemize}
\end{thm}
\begin{proof}(i) For any connected graph $G$, from Corollary \ref{doub2} it follows that $mhs_{\le}(G) \le \psi(G)$,
i.e. $(mhs_{\le} - \psi)(n) \le 0$.
Since, by Corrolary \ref{psit} and Proposition \ref{hitst}, $\psi(P_n)=mhs_{\le}(P_n)=2$ it follows $(mhs_{\le} - \psi)(n) =0$.  \\
(ii) From Proposition \ref{dou1a} and Lemma \ref{hitl1} it is evident that for every graph $G$ it holds $\psi(G) - mhs_{\le}(G) \le n-3$.
By Proposition \ref{dou2a} and Proposition \ref{hitst} this upper bound is reached for complete graph $K_n$.
\end{proof}

Since $mhs_{\le}(G))$ is a lower bound for $\psi(G)$ and
$mhs_<(G)$ represents a lower bound for $\beta_M(G)$ 
it is interesting to find extremal differences between these lower bounds. 
Their exact values are given by Theorem \ref{hslel1}. 

\begin{thm} \label{hslel1} For $n \ge 3$ it holds 
\begin{itemize}
\item [(i)]  $(mhs_{\le} - mhs_<)(n) =0$
\item [(ii)] $(mhs_< - mhs_{\le})(n) = n-2$
\end{itemize}
\end{thm} 
\begin{proof} (i) By Lemma \ref{hitl1} it holds $(mhs_{\le} - mhs_<)(n) \le 0$. 
By Proposition \ref{hitst} this upper bound is reached for e.g. path $P_n$,
which implies \\ $(mhs_{\le} - mhs_<)(n) =0$.\\    
(ii) From Lemma \ref{hitl1} it follows $mhs_<(G) \le n$ and $mhs_{\le}(G) \ge 2$
which implies that $mhs_<(G) - mhs_{\le}(G) \le n-2$.
Acording to Proposition \ref{hitst} this upper bound is reached for complete graph $K_n$. 
\end{proof}

Since both invarinats $mhs_{\le}(G)$ and $mhs_<(G)$ are based
on the minimal hitting set problem and their definitions are similar,
it is expected that there exist classes of graphs for which they have 
the same value. However, Theorem \ref{hslel1} shows that values of these
invariants can drastically differ.   

In \cite{mil21} it is shown that $mhs_<(G)$ is a lower bound of $\beta_M(G)$.
Therefore, it is interesting to see what extremal differences of them are.
This task is solved in Theorem \ref{mdiml1}, which uses Theorem \ref{mngt2}
and Corollary \ref{mngt2a}.

\begin{thm} \label{mngt2} Let $G$ be any connected graph of order $n$. Then $mhs_<(G) = n$ if and only if 
$G$ is a maximal neighbour graph. \end{thm} 
\begin{proof}
($\Rightarrow$) If $mhs_<(G) = n$ from Corrolary \ref{mmetd2} it follows $\beta_M(G) = n$ which by 
Theorem \ref{mngt} implies that $G$ is a maximal neighbour graph.\\
($\Leftarrow$)Let $G$ be a maximal neighbour graph and $v \in V(G)$ be an arbitrary vertex. Then there exists vertex $u \in N(v)$ such that $N(v) \subseteq N(u)$. 
Let us prove that $W_{vu} = \{w | d(v,w) < d(u,w)\} = \{v\}$. Let $z$ be an arbitrary vertex of $G$ different from $v$.
If $z \in N(v)$ then $d(v,z)=1=d(u,z)$ and hence $z \notin W_{vu}$. \\
If $z \in V(G) \setminus N(v)$ then $d(v,z)=1+d(x,z)$ where $x \in N(v)$. \\
If $x = u$ then $d(v,z)=1+d(u,z)$ and $z \notin W_{vu}$. \\
If $x \ne u$ then $d(v,z)=1+d(x,z)=d(u,z)$ and we again conclude that $z \notin W_{vu}$. \\
Since $W_{vu} = \{v\}$ for each $v \in V(G)$ the hitting set of the family \\ $\{W_{uv}, W_{vu} | uv \in E(G) \}$
has to contain all nodes from $V(G)$ and therefore $mhs_<(G)=n$.
\end{proof}

Theorem 1 and Theorem 5 imply the following corollary.

\begin{cor} \label{mngt2a} Let $G$ be any connected graph of order $n$. Then $mhs_<(G) = n$ if and only if 
$\beta_M(G) = n$. \end{cor} 

\begin{thm} \label{mdiml1} For $n \ge 3$ it holds 
\begin{itemize}
\item [(i)] $(mhs_< - \beta_M)(n) =0$
\item [(ii)] $(\beta_M - mhs_<)(n) = n-3$
\end{itemize}
\end{thm} 
\begin{proof} (i) For any connected graph $G$, from Corollary \ref{mmetd2} it follows that $mhs_<(G) - \beta_M(G) \le 0$.
Since, by Proposition \ref{mpath} and Proposition \ref{hitst}, $\beta_M(P_n)=mhs_<(P_n)=2$ it follows $(mhs_< - \beta_M)(n) =0$. \\
(ii) By Proposition \ref{mpath} it holds $\beta_M(G) \le n$ and by Lemma \ref{hitl1} it follows $mhs_<(G) \ge 2$.
Case 1. $\beta_M(G) = n$. \\
According to Corolary \ref{mngt2a} it holds $mhs_<(G) = n$ and consequently it is $\beta_M(G) - mhs_<(G) = 0$. \\
Case 2. $\beta_M(G) \le n-1$. \\
In this case it holds $\beta_M(G) - mhs_<(G) \le n-3$. \\
Since $n \ge 3$ in both cases it holds $\beta_M(G) - mhs_<(G) \le n-3$.
This upper bound is reached for complete bipartite graphs $K_{2,n-2}$
since by Proposition \ref{mbip} from \cite{mdim1} it is $\beta_M(K_{2,n-2}) = n-1$
and by Proposition \ref{hitst} it is $mhs_<(K_{2,n-2}) = 2$.
\end{proof}

Finally, it is interesting to find extremal diferences between
$\psi(G)$ and $\beta_E(G)$. 

For $n=3$ this task is completely resolved by Property \ref{dedge3}. 
For $n \ge 4$, this task is partially resolved 
by Theorem \ref{dedge}, which gives upper and lower bounds
for extremal difference between $\psi(G)$ and $\beta_E(G)$. 

\begin{prp} \label{dedge3} $(\psi - \beta_E)(3) = 1$ and $(\beta_E - \psi)(3) = 0$\end{prp}
\begin{proof} There exist only two connected graphs of order 3: path $P_n$ and cycle $C_n$.
Since $\beta_E(P_3)=1$ and $\beta_E(C_3)=\psi(P_3) = \psi(C_3)=2$ then $(\psi - \beta_E)(3) = 1$. 
\end{proof}

\begin{thm} \label{dedge} For $n \ge 4$ it holds $\lfloor \frac{n}{2} \rfloor - 1 \le (\psi - \beta_E)(n) \le n-3$\end{thm}
\begin{proof} Let $m = \lfloor \frac{n}{2} \rfloor$ and let $T'_n$ be a tree with $V(T'_n) = \{v_1,...,v_n\}$ and 
$E(T'_n) = \{v_iv_{i+1}| 1 \le i \le n-m\} \bigcup \{v_iv_{n-m+i}|2 \le i \le m\}$ \cite{mdim3,mil23}. Figure \ref{ston} illustrates trees 
$T'_n$ for odd and even $n$. From Proposition \ref{tree} it follows that $\psi(T'_n) = l(T'_n) = m+1$. It is proved in \cite{mdim3,mil23}
that $\beta_E(T'_n)=2$ with edge metric base $\{v_1,v_{n-m+1}\}$. Therefore $\psi(T'_n) - \beta_E(T'_n)=m-1=\lfloor \frac{n}{2} \rfloor - 1$
implying $(\psi - \beta_E)(n) \ge \lfloor \frac{n}{2} \rfloor - 1$. Since for any graph $G$ it holds $\psi(G)\le n-1$ and $\beta_E(G) \ge 1$ it follows that 
$(\psi - \beta_E)(n) \le n-2$. \\
Furthermore, this bound can be improved using a similar argument as in \cite{mdim3,mil23}. \\
Since $G$ is connected graph of order $n \ge 3$ it follows that maximal degree $\Delta(G) \ge 2$. We consider two cases:\\
Case 1. $\Delta(G) = 2$. The only such graphs of order $n$, $n \ge 3$ are path $P_n$ and cycle $C_n$. Since by 
Corollary \ref{psit} it holds $\psi(P_n)=2$ and by Proposition \ref{edim1} from \cite{edim1} $\beta_E(P_n)=1$ it 
follows that $\psi(P_n)-\beta_E(P_n)=1$. Also, by Proposition \ref{dcyc} from \cite{cac07} and Proposition \ref{edim1}
from \cite{edim1} it follows that $\psi(C_n)-\beta_E(C_n) \le 1$. In both cases the difference for $n \ge 4$ is less or equal than $n-3$.\\
Case 2. $\Delta(G) \ge 3$. From Proposition \ref{edim2} (\cite{edim1}) $\beta_E(G) \ge \lceil log_2 \Delta(G) \rceil \ge \lceil log_2 3 \rceil = 2$.
As $\psi(G) \le n-1$ it follows $\psi(G) - \beta_E(G) \le n-3$.
\end{proof}

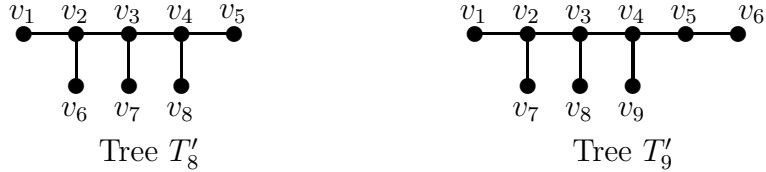
\begin{figure}[htbp]
\begin{center}
\centering\setlength\unitlength{1mm}
\begin{picture}(100,25)
\thicklines
\put(12,0){Tree $T'_8$}
\put(2,17){\circle*{2}}
\put(2,17){\line(1,0){7}}
\put(0,19){$v_1$}
\put(9,17){\circle*{2}}
\put(9,17){\line(1,0){7}}
\put(7,19){$v_2$}
\put(9,10){\circle*{2}}
\put(9,10){\line(0,1){7}}
\put(7,6){$v_6$}
\put(16,17){\circle*{2}}
\put(16,17){\line(1,0){7}}
\put(14,19){$v_3$}
\put(16,10){\circle*{2}}
\put(16,10){\line(0,1){7}}
\put(14,6){$v_7$}
\put(23,17){\circle*{2}}
\put(23,17){\line(1,0){7}}
\put(21,19){$v_4$}
\put(23,10){\circle*{2}}
\put(23,10){\line(0,1){7}}
\put(21,6){$v_8$}
\put(30,17){\circle*{2}}
\put(28,19){$v_5$}

\put(75,0){Tree $T'_9$}
\put(62,17){\circle*{2}}
\put(62,17){\line(1,0){7}}
\put(60,19){$v_1$}
\put(69,17){\circle*{2}}
\put(69,17){\line(1,0){7}}
\put(67,19){$v_2$}
\put(69,10){\circle*{2}}
\put(69,10){\line(0,1){7}}
\put(67,6){$v_7$}
\put(76,17){\circle*{2}}
\put(76,17){\line(1,0){7}}
\put(74,19){$v_3$}
\put(76,10){\circle*{2}}
\put(76,10){\line(0,1){7}}
\put(74,6){$v_8$}
\put(83,17){\circle*{2}}
\put(83,17){\line(1,0){7}}
\put(81,19){$v_4$}
\put(83,10){\circle*{2}}
\put(83,10){\line(0,1){7}}
\put(81,6){$v_9$}
\put(90,17){\circle*{2}}
\put(90,17){\line(1,0){7}}
\put(88,19){$v_5$}
\put(97,17){\circle*{2}}
\put(97,19){$v_6$}

\end{picture}
\caption{Trees $T'_8$ and $T'_9$.}
\label{ston}
\end{center}
\end{figure}

\section{Conclusions}

This paper defines a new graph invariant $mhs_{\le}(G)$, which is 
a new lower bound for $\psi(G)$.  
Exact values of the new invariant are obtained for some special classes
of graphs. Next, some tight bounds for the new invariant in general case
are derived. Finally, some extremal differences between 
several related invariants are obtained.

Direction of future work could be focused to find exact values
of the new invariant for some other interesting classes of graphs
and consider other extremal differences.

\section*{References}
\bibliographystyle{elsarticle-num}
 \bibliography{paper}

\end{document}